\documentclass[11pt]{amsart}

\usepackage[latin1]{inputenc}
\usepackage{amsmath,amsfonts,amssymb}
\usepackage{amsmath}
\usepackage{amssymb}
\usepackage{amsthm}
\usepackage{latexsym}
\def\NZQ{\mathbb}               
\def\NN{{\NZQ N}}

\def\opn#1#2{\def#1{\operatorname{#2}}} 
\opn\chara{char} \opn\length{\ell} \opn\pd{pd} \opn\rk{rk}
\opn\projdim{proj\,dim} \opn\injdim{inj\,dim} \opn\rank{rank}
\opn\depth{depth} \opn\codepth{codepth} \opn\grade{grade}
\opn\height{height} \opn\embdim{emb\,dim} \opn\codim{codim}

\opn\Tr{Tr} \opn\bigrank{big\,rank}
\opn\superheight{superheight}\opn\lcm{lcm}
\opn\trdeg{tr\,deg}%
\opn\reg{reg} \opn\lreg{lreg} \opn\skel{skel} \opn\Gr{Gr}
\opn\dim{dim} \opn\arithdeg{arithdeg}

\opn\div{div} \opn\Div{Div} \opn\cl{cl} \opn\Cl{Cl}
\opn\Spec{Spec} \opn\Supp{Supp} \opn\supp{supp} \opn\Sing{Sing}
\opn\Ass{Ass}
\opn\Ann{Ann} \opn\Rad{Rad} \opn\Soc{Soc}
\opn\Sym{Sym} \opn\Ker{Ker} \opn\Coker{Coker} \opn\Im{Im}
\opn\Hom{Hom} \opn\Tor{Tor} \opn\Ext{Ext} \opn\End{End}
\opn\Aut{Aut} \opn\id{id} \opn\ini{in} \opn\tr{tr}

\opn\nat{nat}\opn\it{it}
\opn\pff{proof}
\opn\Pf{proof} \opn\GL{GL} \opn\SL{SL} \opn\mod{mod} \opn\ord{ord}
\opn\aff{aff} \opn\con{conv} \opn\relint{relint} \opn\st{st}
\opn\lk{lk} \opn\cn{cn} \opn\core{core} \opn\vol{vol}
\opn\link{link} \opn\star{star} \opn\skel{skel}
\opn\gr{gr}

\def\pot#1#2{#1[\kern-0.28ex[#2]\kern-0.28ex]}

\opn\dirlim{\underrightarrow{\lim}}
\opn\inivlim{\underleftarrow{\lim}}
\let\to=\rightarrow

\def\Implies{\ifmmode\Longrightarrow \else
     \unskip${}\Longrightarrow{}$\ignorespaces\fi}
\def\implies{\ifmmode\Rightarrow \else
     \unskip${}\Rightarrow{}$\ignorespaces\fi}
\def\iff{\ifmmode\Longleftrightarrow \else
     \unskip${}\Longleftrightarrow{}$\ignorespaces\fi}

\let\:=\colon

\newtheorem{thm}{Theorem}[section]
\newtheorem{cor}[thm]{Corollary}
\newtheorem{lem}[thm]{Lemma}
\newtheorem{prop}[thm]{Proposition}
\theoremstyle{definition}
\newtheorem{defn}[thm]{Definition}
\theoremstyle{remark}
\newtheorem{rem}[thm]{Remark}
\newtheorem*{ex}{Example}
\numberwithin{equation}{section}
\newtheorem{Claim}[thm]{Claim}

\let\epsilon\varepsilon
\let\phi=\varphi
\let\kappa=\varkappa
\def\qed{\ifhmode\textqed\fi
   \ifmmode\ifinner\quad\qedsymbol\else\dispqed\fi\fi}
\def\textqed{\unskip\nobreak\penalty50
    \hskip2em\hbox{}\nobreak\hfil\qedsymbol
    \parfillskip=0pt \finalhyphendemerits=0}
\def\dispqed{\rlap{\qquad\qedsymbol}}

\opn\Gin{Gin}

\def\FF{{\mathcal F}}

\opn\inii{in} \opn\inim{inm} \opn\rate{rate}
\begin{document}

\title[Binomial edge ideals with quadratic Gr\"{o}bner bases]
 {Binomial edge ideals with quadratic\\ Gr\"{o}bner bases}


\author[Marilena Crupi]{Marilena Crupi}
\address[Marilena Crupi]{Dipartimento di Matematica,
Universit\`a di Messina\\ Viale Ferdinando Stagno d'Alcontres, 31\\ 98166 Messina, Italy.}
\email{mcrupi@unime.it}

\author[Giancarlo Rinaldo]{Giancarlo Rinaldo}
\address[Giancarlo Rinaldo]{Dipartimento di Matematica\\
Universit\`a di Messina\\ Viale Ferdinando Stagno d'Alcontres, 31\\ 98166
Messina, Italy}
\email{giancarlo.rinaldo@tiscali.it}





\begin{abstract}
 We prove that  a binomial edge ideal of a graph $G$ has a quadratic Gr\"{o}bner basis with respect to some term order if and only if the graph  $G$ is closed  with respect to a given labelling of the vertices (\cite{HH}). We also state some criteria for the closedness of a graph $G$ that do not depend necessarily from the labelling of its vertex set.
\end{abstract}

\maketitle

\section*{Introduction}

In this article a graph $G$ means a simple graph without isolated vertices, loops and multiple edges. Let $V(G)=[n] =\{1, \ldots, n\}$ denote the set of vertices and $E(G)$ the set of edges.

One of the main objects of study in combinatorial commutative algebra
is the edge ideal of a graph $G$ which is generated by the monomials $x_ix_j$, where $\{i, j\}$ is an edge of $G$, in 
the polynomial ring $K[x_1,\ldots,x_n]$ over the field $K$. Edge ideals of a graph has been introduced by Villarreal in 1990, \cite{Vi2}, where he studied the Cohen-Macaulay property of such ideals. Many authors have focalized their attention on such ideals (see for example \cite{SVV}, \cite{HH1},\cite{SF1}, \cite{CRT}). 

In 2010, binomial edge ideals were introduced in \cite{HH} and appear independently, but at the same time, also in \cite{MO}. Let $S = K[x_1,\cdots, x_n, y_1,\cdots, y_n]$ be the polynomial ring in $2n$ variables with coefficients in a field $K$. For $i < j$, set $f_{ij} = x_iy_j - x_jy_i$. The ideal $J_G$ of $S$ generated by the binomials $f_{ij} = x_iy_j - x_jy_i$ such that $i<j$ and $\{i,j\}$ is an edge of $G$, is called \textit{the binomial edge ideal} of $G$.

 Such class of ideals is a natural generalization of the ideal of $2$-minors of a $2\times n$-matrix of indeterminates. Really,
the ideal of $2$-minors of a $2\times n$-matrix may be considered as the binomial edge ideal of a complete graph on $[n]$.
Moreover the binomial edge ideal of a line graph, that can be interpreted as an ideal of adjacent minors, has been examined in \cite{DES}.
The importance of such class of binomial edge ideals for algebraic statistics is unquestionable \cite{HH}. Indeed these ideals arise naturally in the study of conditional independence statements \cite{DSS}. Many algebraic properties of binomial edge ideals in terms of properties of the underlying graph were studied in \cite{HH} and \cite{HEH}.

In \cite{HH}, Theorem 1.1, the authors proved the following:
\begin{thm}\label{th:HH}
 Let $G$ be a graph on the vertex set [n], and let $<$ the lexicographic order induced by $x_1>\cdots>x_n>y_1>\cdots>y_n$ on $S$. Then the following conditions are equivalent:
\begin{enumerate}
 \item The generators $f_{ij}$ of $J_G$ form a quadratic Gr\"obner basis.
 \item For all edges $\{i,j\}$ and $\{k, \ell\}$ with $i<j$ and $k<\ell$ one has $\{j,\ell\}\in E(G)$ if $i=k$, and $\{i,k\}\in E(G)$ if $j=\ell$.
\end{enumerate}
\end{thm}
The authors in \cite{HH}, called a graph $G$ on $[n]$ \textit{closed with respect to the given labelling of the vertices} if $G$ satisfies condition (2). The term closed graph is not standard terminology in graph theory. Nevertheless this class of graphs is related to a well-know class of graphs: the chordal graphs. A closed graph is chordal (\cite{HH}) but the converse is not true. Indeed a closed graph is a claw-free chordal graph, where for a claw we mean a graph with three different edges $e_1,e_2,e_3$ such that  $e_1\cap e_2 \cap e_3 \neq \emptyset$.

In Theorem \ref{th:HH} the role of the lexicographic order on $S$ is fundamental.
In this article we are able to state that the existence of a quadratic Gr\"{o}bner basis for $J_G$ is not related to the lexicographic order on $S$. In fact, one of the main result in the paper assures
that the closed graphs are the only graphs for which the binomial edge ideal $J_G$ has a quadratic Gr\"{o}bner basis with respect to some
term order on $S$ (Theorem \ref{th:qua}).  Our result underlines also the relation between binomial edge ideals and edge ideals. In fact as a consequence we obtain that $J_G$ has a quadratic Gr\"obner basis with respect to some term order $\prec$ on $S$ if and only if  $\ini(J_G)$ is the edge ideal of a bipartite graph with bipartition $V_1 =\{x_1,\cdots, x_n\}$ and $V_2 =\{y_1,\cdots, y_n\}$. The strict relation between algebraic invariants of an ideal $J$ and $\ini(J)$ is well known (see for example \cite{DE}, Chapter 15).
\par\noindent


Furthermore Theorem \ref{th:HH} and Theorem \ref{th:qua} suggest that it would be interesting to state some criteria for the closedness of a simple graph $G$. Since the characterizations of closed graphs $G$ (see \cite{HH}, \cite{HEH}) depends on the labelling of $V(G)$, our aim is to state some new criteria for the closedness of a graph that do not depend necessarily on the labelling of its vertex set (Theorem \ref{th:po}
and Corollary \ref{Cor1}).

We believe that by an ordering on the vertices obtained by lexicographic breadth first search and a right specialization of the algorithm on chordality test (see Algorithms 2, 3 of \cite{Ha} or \cite{RLT}), it is possible to test the closedness of a graph as a consequence of Theorem \ref{th:po} in linear time. But this is not the aim of this paper.

The paper is organized as follows.

Section 1 contains some preliminaries and notions that we will use in the paper.

In Section 2, we state a fundamental result that gives the motivation of an intensive study of closed graphs: we prove that the only graphs having quadratic Gr\"obner basis with respect to a given monomial order are the closed ones (Theorem \ref{th:qua}). The statement is obtained by the construction of a special oriented graph (Definition \ref{Def:or}).

In Section 3, we introduce the notion of a linear quasi-tree simplicial complex (Definition \ref{def:new}) and we relate it with a closed graph (Proposition \ref{pro:lqt}). Moreover we give a characterization of the closedness of a graph $G$ in terms of particular cliques of $G$ (Proposition \ref{proposition:newdef}).
This result will be crucial in the sequel.

In Section 4, we analyze the behaviour of the set of facets $\mathcal F (\Delta(G))$ of the clique complex $\Delta(G)$ (Definition \ref{def:cli}) of a graph G when $\Delta(G)$ is a linear quasi-tree (Proposition \ref{pro:int}). We introduce a special subclass of the linear quasi-tree complexes: the class of \textit{closed} complexes (Definition \ref{def:closed}). The section contains the main results in the paper. We give a criterion for the closedness of a graph $G$ that is independent from the labelling of $V(G)$ (Theorem \ref{th:po}). We show that a graph $G$ is closed if and only if the clique complex $\Delta(G)$ is a closed complex (Corollary \ref{Cor1}).



\section{Preliminaries}\label{sec:pre}
In this section we recall some concepts and a notation on graphs and on simplicial
complexes that we will use in the article.

Let $G$ be a simple graph with vertex set $V(G)$ and the edge set $E(G)$. Let $v, w \in V(G)$. A path $\pi$ from $v$ to $w$ is a sequence of vertices $v=v_0, v_1, \cdots, v_t=w$ such that $\{v_i, v_{i+1}\}$ is an edge of  the underlying graph. A graph $G$ is \textit{connected} if for every pair of vertices $v_1$ and $v_2$ there is a path from $v_1$ to $v_2$. If $G$ is \textit{directed} (or \textit{digraph}), that is, $G$ consists of a finite nonempty set of vertices  with a prescribed collection $X$ of ordered pairs of distinct vertices, then the path is called \textit{directed}, if either $(v_i, v_{i+1})$ is an arrow for all $i$, or $(v_{i+1}, v_i)$ is an arrow for all $i$.


When we fix a given labelling on the vertices we say that $G$ is a graph on $[n]$. \\
Let $G$ be a graph with vertex set $[n]$. A subset $C$ of $[n]$ is called a \textit{clique} of $G$ is for all $i$ and $j$ belonging to $C$ with $i \neq j$ one has $\{i, j\} \in E(G)$.

Set $V = \{x_1, \ldots, x_n\}$. A \textit{simplicial complex}
$\Delta$ on the vertex set $V$ is a collection of subsets of $V$
such that
\begin{enumerate}
\item[(i)] $\{x_i\} \in \Delta$  for all $x_i \in V$ and
\item[(ii)] $F \in \Delta$ and $G\subseteq F$ imply $G \in \Delta$.
\end{enumerate}
An element $F \in \Delta$ is called a \textit{face} of $\Delta$. For
$F \in \Delta$ we define the \textit{dimension} of $F$ by $\dim F
= |F| -1$, where $|F|$ is the cardinality of the set
$F$. A maximal face of $\Delta$  with respect to inclusion is
called a \textit{facet} of $\Delta$.

If $\Delta$ is a simplicial complex  with facets $F_1, \ldots, F_q$, we call $\{F_1, \ldots, F_q\}$ the facet set of $\Delta$ and we denote it by $\FF(\Delta)$. When $\FF(\Delta)=\{F_1, \ldots, F_q\}$, we write $\Delta=\langle F_1, \ldots, F_q\rangle$.

\begin{defn}\label{def:cli}
The \textit{clique complex} $\Delta(G)$ of $G$ is the simplicial complex whose faces are the cliques of $G$.
\end{defn}

\begin{defn} \label{lb} Let $\Delta$ be a simplicial complex. A facet $F\in \FF(\Delta)$ is said to be a \textit{leaf} of $\Delta$ if either $F$ is the only facet of $\Delta$, or there exists a facet $B\in \FF(\Delta)$, $B \neq F$, called a {\em branch} of $F$, such that $H\cap F\subseteq B\cap F$ for all $H\in \FF(\Delta)$ with $H\neq B$.
\end{defn}
Observe that for a leaf $F$ the subcomplex $\Delta'$ with  $\FF(\Delta')=\FF(\Delta)\setminus F$ coincides with the restriction
$\Delta_{[n]\setminus(F\setminus (B\cap F))}$.

We finish this section by recalling the following definition from \cite{JT}.

\begin{defn}\label{def:qt} Let $\Delta$ be a simplicial complex. $\Delta$ is called a quasi-forest if there exists a labelling $F_1, \cdots, F_q$ of the facets of $\Delta$,
such that for every $1 < i \leq q$, the facet $F_i$ is a leaf of the subcomplex $\langle F_1, \cdots, F_i\rangle$.
The sequence $F_1,\ldots,F_q$ is called a leaf order of the quasi-tree. A connected quasi-forest is called a quasi-tree.
\end{defn}

\section{Quadratic Gr\"obner bases}
In this section we observe that the only graphs having quadratic Gr\"obner bases  with respect to a monomial order $\prec$ are the  \textit{closed graphs} with respect to a labelling induced by $\prec$.

Let $G$ be a graph on the vertex set $V(G) = [n]$, $E(G)$ its edge set and $S = K[x_1,\cdots, x_n,y_1,\cdots, y_n]$.

%
%
%

%
%


\begin{defn}\label{Def:or}
Let $J_G$ be the binomial edge ideal of $G$ and let $\prec$ a term order on $S$. We define an oriented graph $G_\prec$ with $V(G_\prec) = V(G)$ and edge set
\[
 E(G_\prec)=\{(i,j):x_i y_j\in \ini_\prec J_G \}.
\]
\end{defn}
\begin{prop}\label{pro:dag}
 $G_\prec$ is an acyclic directed graph.
\end{prop}
\begin{proof}
 It is sufficient to show that every cycle in $G$ is not a directed cycle in $G_\prec$.
 Let
\[
 \{i_1,i_2,\ldots, i_r\}\subseteq V(G)
\]
be the vertices of a cycle and suppose that $(i_j,i_{j+1})\in E(G_\prec)$ for $j=1,\ldots,r-1$. We will show that $(i_r,i_1)\not\in E(G_\prec)$.

By hypothesis we have that $x_{i_j} y_{i_{j+1}} \succ x_{i_{j+1}} y_{i_j}$ for $j=1,\ldots,r-1$.
Since $\prec$ is a term order, then $y_{i_3}(x_{i_1} y_{i_2})\succ y_{i_3}(x_{i_2} y_{i_1})$ and $y_{i_1}(x_{i_2} y_{i_3})\succ y_{i_1}(x_{i_3} y_{i_2})$. Therefore $y_{i_3}(x_{i_1} y_{i_2})\succ y_{i_1}(x_{i_3} y_{i_2})$ and $x_{i_1} y_{i_3}\succ x_{i_3} y_{i_1}$.

By the same argument we have that $y_{i_4}(x_{i_1} y_{i_3})\succ y_{i_4}(x_{i_3} y_{i_1})$ and $y_{i_1}(x_{i_3} y_{i_4})\succ y_{i_1}(x_{i_4} y_{i_3})$. Hence $y_{i_4}(x_{i_1} y_{i_3})\succ y_{i_1}(x_{i_4} y_{i_3})$ and $x_{i_1} y_{i_4}\succ x_{i_4} y_{i_1}$, and so on. Finally, we will have that $x_{i_1} y_{i_r}\succ x_{i_r} y_{i_1}$.
\end{proof}
\begin{rem}\label{rem:multideg}\rm
 We observe that the ideal $J_G$ of $S$ is multigraded if we assign the following multidegrees to the indeterminates of $S$:
\[
 \deg(x_i)=\deg(y_i)=(0,\ldots,0,1,0,\ldots,0)\in \NN^n,
\]
where the entry $1$ is at the $i$-th position. Hence the only binomials of degree $2$ in $J_G$ are the generators of $J_G$ up to scaling.
\end{rem}
\begin{thm}\label{th:qua}
Let $G$ be a graph. The following conditions are equivalent:
\begin{enumerate}
 \item $G$ is closed on $[n]$;
 \item $J_G$ has a quadratic Gr\"obner basis with respect to some term order $\prec$ on $S$.
\end{enumerate}
\end{thm}
\begin{proof}
 $1.\Rightarrow 2$. See \cite{HH}, Theorem 1.1.\\
$2.\Rightarrow 1$. By Proposition \ref{pro:dag} $G_\prec$ is a directed acyclic graph. Hence there exists a labelling
\[
 \omega:V(G_\prec) \to [n]
\]
such that for all $(i,j)\in E(G_\prec)$ we have that $\omega(i)<\omega(j)$. This means that $\omega$ is compatible with the orientation of $G_\prec$ (see for example \cite{BJG}, Proposition 1.4.3).

We will show that the graph $G$ is closed with respect to the  labelling $\omega$.

Let $i_1,i_2,i_3\in V(G_\prec)$ such that $\omega(i_1)=i$, $\omega(i_2)=j$, $\omega(i_3)=k$ and let $\{i_1,i_2\}$, $\{i_1,i_3\}\in E(G)$. It follows that $\{i,j\}$, $\{i,k\}$ are edges of $G$ with respect to the labelling $\omega$. By condition (2) of Theorem \ref{th:HH}, we have to analyze the following two cases:
\begin{enumerate}
 \item[(a)] $i<j$, $i<k$;
 \item[(b)] $i>j$, $i>k$.
\end{enumerate}
Case (a). Since $\omega$ is compatible with the oriented graph $G_\prec$, we have the following inequalities
\begin{equation}\label{inequalities}
 x_{i_1} y_{i_2}\succ x_{i_2} y_{i_1} \mbox{ and } x_{i_1} y_{i_3}\succ x_{i_3} y_{i_1}.
\end{equation}
By hypothesis the $S$-polynomial
\[
 S(f_{i_1i_2},f_{i_1i_3})=y_{i_1} f_{i_2i_3}=y_{i_1}(x_{i_2} y_{i_3} -x_{i_3}y_{i_2})
\]
reduces to $0$. Therefore exists a binomial $x_{i_s} y_{i_t}-x_{i_t} y_{i_s}\in J_G$ (see Remark \ref{rem:multideg}) whose leading monomial divides the leading monomial of $y_{i_1} f_{i_2i_3}$. Suppose that $\ini(f_{i_1i_2})=x_{i_2}y_{i_1}$. This contradicts the first inequality in (\ref{inequalities}). By the same argument and the second inequality in (\ref{inequalities}), $\ini(f_{i_1i_3})$ does not divide $\ini(y_{i_1} f_{i_2i_3})$. Hence $f_{i_2i_3}\in J_G$ and $\{j, k\}$ is an edge of $G$ with respect to the labelling $\omega$. Case (b) follows by similar arguments.

%

\end{proof}

\section{Closed graphs and linear quasi-tree complexes}\label{sec:closed}
In this section we introduce the notion of a simplicial complex which is a linear quasi-tree. This class of simplicial complexes is a subclass of the quasi-forest complexes (Definition \ref{def:qt}). Our aim is to underline the close link that there exists between the closed graphs and these simplicial complexes. First of all we recall the following definition \cite{HEH}, Definition 2.1.

\begin{defn} A graph $G$ is \textit{closed} if there exists a labelling for which it is closed.
\end{defn}

%
%

We quote the next result from \cite{HEH}, Theorem 2.2.

\begin{thm} \label{th:inter} Let $G$ be a graph on [n]. The following conditions are equivalent:
\begin{enumerate}
\item[(1)] $G$ is closed;
\item[(2)] there exists a labelling of $G$ such that all facets of $\Delta(G)$ are intervals $[a,b] \subseteq [n]$.
\end{enumerate}
Moreover, if the equivalent conditions hold and the facets $F_1, \ldots, F_r$ of $\Delta(G)$ are labeled such that $\min(F_1) < \min(F_r) < \cdots < \min (F_q)$, then $F_1, \ldots, F_r$ is a leaf order of $\Delta(G)$.
\end{thm}

Since a  graph is closed if and only if each connected component is closed we assume from now on that the graph $G$ is connected.

Thanks to Theorem \ref{th:inter} if $G$ is a closed graph on the vertex set $[n]$ and $\Delta(G)$ is the clique complex, then we may assume that

\begin{equation}\label{inter}
    \Delta(G) =
\langle [m_1,M_1], [m_2,M_2],\ldots,[m_r,M_r]\rangle,
\end{equation}
with $1=m_1<m_2<\ldots<m_r<n$, $1<M_1<M_2<\ldots<M_r=n$  with $m_i < M_i$ and $m_{i+1} \leq M_i$, for $i \in [r]$.

Now we introduce a special class of the quasi-trees complexes.
\begin{defn}\label{def:new}
A simplicial complex is a linear quasi-tree if there exists an order  on the facets
\[
 F_1,\ldots,F_q
\]
such that
\begin{enumerate}
 \item $F_i$ is a leaf for the subcomplex $\langle F_i,\ldots,F_q\rangle$;
 \item $F_{i+1}$ is the only branch of $F_i$ for all $i<q$.
\end{enumerate}
\end{defn}

\begin{rem}\rm Let $\Delta$ be a simplicial complex and let $\FF(\Delta)=\{F_1, \ldots, F_q\}$ be the set of its facets. It is always possible to verify if $\Delta$ is a linear quasi tree and in the positive case it is possible to order $\FF(\Delta)$ so that conditions (1) and (2) of Definition \ref{def:new} are satisfied.\\
In fact, if $\Delta$ is a linear quasi tree, then there exists a leaf $F_i$, that is a facet of $\Delta$ satisfying Definition \ref{lb}. In order to determine $F_i$ it is sufficient to intersect the facet $F_i$, $i = 1, \ldots, q$, with the other facets. Let $F_{i_1}$ be such a facet and let $F_{i_2}$ be its branch. It must be unique by (2) of Definition \ref{def:new}. \\
If $F_{i_1}$ is a leaf and $F_{i_2}$ is its unique  branch, then we consider the subcomplex $\Delta' = \FF(\Delta) \setminus \{F_{i_1}\}$ and we verify if $F_{i_2}$ is a leaf of $\Delta'$ and if its branch is unique and so on. Proceeding in this way we will obtain a linear order $F_{i_1},F_{i_2},\ldots, F_{i_q}$ with respect to which $\Delta$ is a linear quasi tree.\\
We will show this process by the next example.
\end{rem}

\begin{ex}\rm Let $\Delta = \langle F_1, F_2, F_3, F_4 \rangle$, with $F_1 = \{a,b,f\}$, $F_2 = \{a,e,f\}$, $F_3 = \{b,c,f\}$ and $F_4 = \{d,e,f\}$. We want to determine a order on the facet set $\FF(\Delta)$ so that $\Delta$ is a linear quasi tree.\\
Consider the facet $F_1$. We have:
\[F_1 \cap F_2 = \{a,f\},\quad F_1 \cap F_3 = \{b,f\}, \quad F_1 \cap F_4 = \{f\}.\]
Since $F_1 \cap F_2$ and $F_1 \cap F_3$ are not comparable, then $F_1$ is not a leaf of $\Delta$ (Definition \ref{lb}).\\
Consider the facet $F_2$. We have:
\[F_2 \cap F_1 = \{a,f\},\quad F_2 \cap F_3 = \{f\}, \quad F_2 \cap F_4 = \{e,f\}.\]
Since $F_2 \cap F_1$ and $F_2 \cap F_4$ are not comparable, then $F_2$ is not a leaf of $\Delta$ (Definition \ref{lb}).\\
Now consider the facet $F_3$. We have:
\[F_3 \cap F_1 = \{b,f\},\quad F_3 \cap F_2 = \{f\}, \quad F_3 \cap F_4 = \{f\}.\]
Hence $F_1$ is the unique branch of $F_3$ and consequently $F_3$ is a leaf of $\Delta$. \\
Now consider the subcomplex of $\Delta$: $\Delta' = \langle F_1, F_2, F_4 \rangle$. We have:
\[F_1 \cap F_2 = \{a,f\},\quad F_1 \cap F_4 = \{f\}.\]
It follows that $F_2$ is the unique branch of $F_1$ and $F_1$ is a leaf of $\Delta'$. It is easy to observe that we can conclude that $\Delta$ is a linear quasi tree with respect to the following order on $\FF(\Delta)$: $F_3, F_1, F_2, F_4$.
\end{ex}

>From now on when we consider a simplicial complex $\Delta$ that is a linear quasi-tree we write $\Delta=\langle F_1,\ldots,F_q\rangle$
with leaf order $\{F_1,F_2,\ldots,F_q\}$ on the facet set. We state the following.
\begin{prop}\label{pro:lqt}
Let $G$ be a graph on [n]. If $G$ is a closed graph, then $\Delta(G)$ is a linear quasi-tree.
\end{prop}
\begin{proof}
>From (\ref{inter}), since $G$ is closed, we may assume $\Delta(G)=\langle F_1,\ldots, F_r\rangle $, where $F_i= [m_i,M_i]$, for $i=1,\ldots,r$.


We observe that $[m_i,M_i]\cap [m_{i+1},M_{i+1}]=[m_{i+1},M_{i}]$.
Since $m_{i+d}> m_{i+1}$ for all $d\geq 2$, then
\[
F_i\cap F_{i+d}=[m_{i+d},M_i]\varsubsetneq [m_{i+1},M_{i}].
\]
Therefore $F_{i}$ is a leaf and $F_{i+1}$ is the unique branch for $F_i$.
\end{proof}





\begin{ex}\rm The converse of Proposition \ref{pro:lqt} is not true. In fact there are linear quasi-trees that are not closed.

Let $V(G)=\{a,b,c,d,e,f\}$ and let $\Delta(G)=\langle F_1, F_2, F_3 \rangle$ be the facet set of its clique complex, where $F_1=\{a,b,c\}$, $F_2=\{b,c,d,e\}$ and $F_3=\{b,e,f\}$. We can easily check that $\langle F_1,F_2,F_3 \rangle$ is a linear quasi-tree but the subgraph induced by the vertices $\{a,b,d,f\}$ is a claw, i.e. the complete bipartite graph $K_{1,3}$. Therefore, by \cite{HH}, Proposition 1.2, $G$ is not closed.
\end{ex}

We finish this section giving a criterion for the closedness of a graph with respect to a given labelling that will be crucial in the sequel.

Let $G$ be a graph on the vertex set $V(G) = [n]$. For each vertex $j\in V(G)$ we define a partition
of its neighborhood $N_G(j) = \{i \in [n]\,:\, \{i,j\}\in E(G)\}$ into two sets as follows:
\[ N_G(j)=N_G^{<}(j)\cup N_G^{>}(j),\]
where
\[
 N_G^{<}(j)=\{i:\{i,j\}\in E(G),i<j\}, \,\,N_G^{>}(j)=\{k:\{j,k\}\in E(G), j<k\}.
\]

\begin{prop}\label{proposition:newdef}
Let $G$ be a graph on $[n]$. The following conditions are equivalent:
\begin{enumerate}
 \item $G$ is closed with respect to the given order of the vertices;
 \item for all vertices $j\in V(G)$ the sets $N_G^{<}(j)$, $N_G^{>}(j)$ are cliques of $G$.
\end{enumerate}
\end{prop}
\begin{proof}
1) $\Rightarrow $ 2) Let $j\in V(G)$. For all $i_1$, $i_2\in N^<_G(j)$, by definition, we have  that $\{i_1,j\}$, $\{i_2,j\}\in E(G)$ with $i_1<j$ and $i_2<j$. Since $G$ is closed, then $\{i_1,i_2\}\in E(G)$. Hence  $N^<_G(j)$ is a clique. Similarly for $N^>_G(j)$.\\
2) $\Rightarrow $ 1) Let $\{j,k_1\}$, $\{j,k_2\}\in E(G)$ with $j<k_1$, $j<k_2$. This implies $k_1,k_2 \in N_G^>(j)$. Since $N_G^>(j)$ is a clique, then $\{k_1,k_2\}\in E(G)$. The other case follows by similar argument.
\end{proof}

\section{Closed graphs with respect to any labelling}
In this section we give a characterization of closed graphs which does not depend on the labelling of their vertex sets. For this reason we study the clique complex $\Delta(G)$ of the simple graph $G$.

Let $\Delta=\langle F_1,\ldots, F_r\rangle $ be a simplicial complex.
We set
\[
F_{i_1,i_2,\ldots,i_s}:=F_{i_1}\cap F_{i_2} \cap \ldots\cap F_{i_s}
\]
with $1\leq i_1<i_2<\ldots<i_s\leq r$ and $F_{i,i}:=F_i$ for $i \in [r]$.

\begin{prop}\label{pro:int} If $\Delta=\langle F_1,\ldots,F_r\rangle$ is a linear quasi-tree, then $F_{i,j}=F_{i,i+1,\ldots,j}$, $1\leq i<j\leq r$. In particular, $F_{k,\ell}\supseteq F_{i,j}$ for all $k$, $\ell$ such that $i\leq k\leq \ell\leq j$.
\end{prop}
\begin{proof}
We proceed by descending induction on $i$, for $i<j$.
If $i=j-1$ there is nothing to prove.
Let $i\leq j-1$ and suppose $F_{i,j}=F_{i,i+1,\ldots,j}$. We have to prove that $F_{i-1, j} = F_{i-1, i, i+1,\ldots,j}$.

Since $F_{i-1,i,j}=F_{i-1}\cap F_{i,j}=F_{i-1, i, i+1,\ldots,j}$, we need to show that
\[
 F_{i-1}\cap F_{i,j}=F_{i-1,j}.
\]
By definition $F_{i-1,i,j}\subseteq F_{i-1,j}$. Since $F_i$ is a branch of $F_{i-1}$, then $F_{i-1,j}\subseteq F_{i-1,i}$. Hence
$F_{i-1,j}\cap F_j \subseteq F_{i-1,i}\cap F_j$, that is, $F_{i-1,j} \subseteq F_{i-1,i,j}$ and the assertion follows.\end{proof}
%
%

Denote by $\mathcal P=\{F_{i,j}:1\leq i\leq j\leq r\}$ the poset whose order is given by the inclusion and set $F_{i,j}=\emptyset$ if either $i<1$ or $j>r$. If $F, G\in \mathcal P$ are not comparable or $F\neq\emptyset$ or $G\neq\emptyset$, we write $F\not\sim G$ .\\
\begin{defn} \label{def:closed} Let $\Delta = \langle F_1,\ldots,F_r\rangle$ be a linear quasi-tree. $\Delta$ is called closed if
the following properties are satisfied:
\begin{enumerate}
  \item[(I)] $F_{i,j}\not\sim F_{k,\ell}$ if $i< k$, $j<\ell$, $i,j, k, \ell \in [r]$ (incomparability);
  \item[(C)] $F_{i+1,i+d}=F_{i,i+d}\cup F_{i+1,i+d+1}$ if $F_{i,i+d+1}\neq \emptyset$ with $d\geq 1$ and $i \in [r]$ (covering).
  \end{enumerate}
\end{defn}
\begin{thm}\label{the:main}
Let $G$ be a graph on [n]. If $G$ is closed, then $\Delta(G)$ is closed.
\end{thm}
\begin{proof}
Since, from Proposition \ref{pro:lqt}, $\Delta(G)$ is a linear quasi-tree, we have only to prove that the facet set
$\FF(\Delta(G)) =\{F_1,\ldots,F_r\}$ satisfies properties (I) and (C) in Definition \ref{def:closed}.

%
(I). Since $G$ is closed on $[n]$, if $F_{i,j}\neq \emptyset$ and $F_{k,\ell}\neq \emptyset$,  from (\ref{inter}) we have:
\[
F_{i,j}= F_i \cap F_{j} = [m_i,M_i]\cap [m_j,M_j]=[m_j,M_i],
\]
\[
F_{k,\ell}= F_k \cap F_{\ell} = [m_k,M_k]\cap [m_{\ell},M_{\ell}]=[m_{\ell},M_k],
\]
with $i<j$ and $k<\ell$. We may assume $i<k$ and $j < \ell$. Hence by (\ref{inter}) $m_j<m_{\ell}$ and $M_i <M_k$. Therefore $M_k\in F_{k,\ell}\setminus F_{i,j}$ and $m_j\in F_{i,j}\setminus F_{k,\ell}$, that is $F_{i,j}\nsim F_{k,\ell}$.


(C). Since $F_{i,i+d+1}\neq \emptyset$ and $G$ is closed, then
\[
F_{i,i+d+1}=[m_{i+d+1},M_i]\neq \emptyset.
\]
Therefore  $m_{i+d+1}\leq M_i$, and
\[
F_{i,i+d}\cup F_{i+1,i+d+1}=[m_{i+d},M_i]\cup [m_{i+d+1},M_{i+1}]=[m_{i+d},M_{i+1}]=F_{i+1,i+d}.
\]\end{proof}

To prove that $\Delta(G)$ closed implies $G$ closed we need a labelling on the vertices of $G$ for which $G$ is closed.





\begin{lem}\label{lem:increasing}
Let $\Delta(G) = \langle F_1,\ldots,F_r \rangle$ be a linear quasi-tree.
Set $n_i=\max\{j: F_{i,j}\neq\emptyset,j\in[r]\}$. Then $n_1\leq n_2\leq \cdots\leq n_r$ and every set $F_{i,j}$ in
\[
 \mathcal{B}=\{F_{1,1},\ldots,F_{1,n_1}, F_{2,n_1},\ldots, F_{2,n_2},\ldots,F_{r,r}\}
\]
is not empty.
\end{lem}
\begin{proof}
Since $F_{i,n_i}\neq\emptyset$, then $F_{i+1,n_i}\neq\emptyset$ (Proposition \ref{pro:int}). Hence $n_{i+1}\geq n_i$. Moreover, by Proposition \ref{pro:int}, we can also state that every set in $\mathcal{B}$ is not empty.
\end{proof}

%
%
%
Now we are in position to state the main result in the paper.

\begin{thm}\label{th:po} Let $G$ be a graph. Suppose that $\Delta(G)$ is  closed. Let $F_1,\ldots,F_r$ be the leaf order of $\Delta(G)$
and consider the family
\[
\mathcal F=\{ F'_{i,j}\}_{F_{i,j}\in \mathcal B},
\]
where $\mathcal{B}$ is defined as in Lemma \ref{lem:increasing} and $F'_{i,j}=F_{i,j}\setminus (F_{i-1,j} \cup F_{i,j+1})$ . Then
\begin{enumerate}
 \item The family $\mathcal{F}$ is a partition of $V(G)$.
 \item $G$ is closed with respect to the following total order on the vertices: For the vertices in each $F_{i,j}'$ we fix an arbitrary total order and set $u<v$, if $u\in F_{i,j}'$ and  $v\in F_{k,\ell}'$ with $i<k$ or $i=k$ and $j<\ell$.
\end{enumerate}
\end{thm}
\begin{proof}
$(1)$. First of all, we prove the following claim.
\begin{Claim} \label{claim} Let $F_{i,j}\neq \emptyset$ then
\[
F_i\cup F_j=\bigcup_{k=i}^j F_k=\left(\bigcup_{k=i}^{j-1} F_{i,k}\right) \cup \left(\bigcup_{k=i+1}^j F_{k,j}\right)\cup F_{i,j}.
\]
\end{Claim}

\noindent
\textit{Proof of the Claim.} Let $i\leq k\leq \ell\leq j$. Since, by assumption  $F_{i,j}\neq \emptyset$, then Proposition \ref{pro:int} implies that  $F_{k,\ell}\neq \emptyset$.\\
By condition (I) in Definition \ref{def:closed} and Proposition \ref{pro:int},
$F_{k,\ell}\subseteq F_{i,j}$ if and only if $1\leq i\leq k\leq \ell \leq j \leq r$. Hence the poset $\mathcal P_{ij}=\{F_{k,\ell}:i\leq k\leq \ell\leq j\}$, whose partial order is given by the inclusion, is the following:

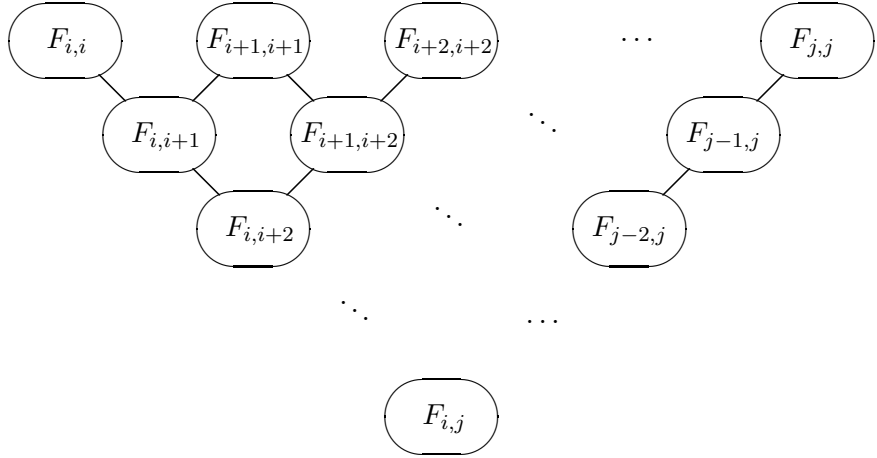
\begin{center}
\setlength{\unitlength}{.25cm}
\begin{picture}(45,30)(-3,-25)

\put(0,0){\oval(6,4)}
\put(10,0){\oval(6,4)}
\put(20,0){\oval(6,4)}
\put(29.5,0){$\ldots$}

\put(40,0){\oval(6,4)}

\put(5,-5){\oval(6,4)}
\put(15,-5){\oval(6,4)}
\put(24.5,-5){$\ddots$}

\put(35,-5){\oval(6,4)}

\put(10,-10){\oval(6,4)}
\put(19.5,-10){$\ddots$}

\put(30,-10){\oval(6,4)}


\put(14.5,-15){$\ddots$}
\put(24.5,-15){$\ldots$}

\put(20,-20){\oval(6,4)}

\put(1.8,-1.8){\line(1,-1){1.41}}
\put(8.2,-1.8){\line(-1,-1){1.41}}

\put(11.8,-1.8){\line(1,-1){1.41}}
\put(18.2,-1.8){\line(-1,-1){1.41}}

\put(38.2,-1.8){\line(-1,-1){1.41}}

\put(6.8,-6.8){\line(1,-1){1.41}}
\put(13.2,-6.8){\line(-1,-1){1.41}}

\put(33.2,-6.8){\line(-1,-1){1.41}}

%
%
%
%
%
%
%
%
%
%
%
%
%
%
%
%

\put(-1,-0.4){$F_{i,i}$}
\put(7.5,-0.4){$F_{i+1,i+1}$}
\put(17.5,-0.4){$F_{i+2,i+2}$}
\put(38.5,-0.4){$F_{j,j}$}

\put(3.5,-5.4){$F_{i,i+1}$}
\put(12.5,-5.4){$F_{i+1,i+2}$}
\put(33,-5.4){$F_{j-1,j}$}

\put(8.5,-10.4){$F_{i,i+2}$}
\put(28,-10.4){$F_{j-2,j}$}

\put(19,-20.4){$F_{i,j}$}
\end{picture}
\end{center}

We observe that $\bigcup_{k=i}^j F_k=\bigcup_{F\in \mathcal{P}_{i,j}} F$. Since $F_{k-1,k+1}\neq \emptyset$, for $k=i+1,\ldots,j-1$, then by condition (C) we have $F_{k,k}=F_{k-1,k}\cup F_{k,k+1}$, that is
\[
 \bigcup_{k=i}^j F_k=\bigcup_{F\in \mathcal{P}'_{i,j}} F
\]
with $\mathcal{P}'_{i,j}=\mathcal{P}_{i,j}\setminus \{F_k:k=i+1,\ldots,j-1\}$.
By similar argument we may subtract all the redundant elements $F_{k,\ell}$ with $i<k<\ell<j$. Hence
\[
 \bigcup_{k=i}^j F_k=\left(\bigcup_{k=i}^{j-1} F_{i,k}\right) \cup \left(\bigcup_{k=i+1}^{j} F_{k,j}\right)\cup F_{i,j}=F_i\cup F_j,
\]
and Claim \ref{claim} is proved.


Let $\mathcal P=\{F_{i,j}:1\leq i\leq j\leq r\}$ be the poset induced by the inclusion. We say that an element $F_{i,j}\in \mathcal P$ is an \textit{inner} element if $F_{i-1,j+1}\in \mathcal P$ and  $F_{i-1,j+1}\neq \emptyset$. Otherwise
an element of $\mathcal P$ is said to be \textit{border} element.\\
We observe that the border elements are exactly  the elements of $\mathcal{B}$ described in Lemma \ref{lem:increasing}, and
\[
 V(G)=\bigcup_{F_{i,j}\in \mathcal B}  F_{i,j}.\hspace{1cm} (*)
\]
In fact if $v\in V(G)$, then $v\in F_{k,k}\in \mathcal F(\Delta(G))$. If $F_{k,k}\in \mathcal B$ we have nothing to prove. Suppose $F_{k,k}\notin \mathcal B$ then $F_{k-1,k+1}\neq \emptyset$ and, since $\Delta(G)$ is closed  by property $(C)$ $F_{k,k}=F_{k-1,k}\cup F_{k, k+1}$. We may assume $v\in F_{k-1,k}$. If $F_{k-1,k}\notin \mathcal B$ applying the same argument after a finite number of steps  we obtain $v\in F_{i,j}\in \mathcal B$.
If we remove the redundant elements in $(*)$ we obtain
\[
 V(G)=\bigcup_{F_{i,j}\in \mathcal B}  F'_{i,j}\hspace{1cm} (**)
\]
where $F'_{i,j}=F_{i,j}\setminus (F_{i-1,j} \cup F_{i,j+1})$. We observe the following
\begin{equation}
\label{statement}
 \mbox{if }v\in F_{i,j}', \mbox{ then }v\in F_k \mbox{ if and only if }k=i,\ldots,j.
\end{equation}
This assertion can be deduced from the structure of the poset $\mathcal P$. For sake of completeness we give a direct proof. Since $v\in F'_{i,j}$ then $v\in F_{i,j}$ and by Proposition \ref{pro:int}, $v\in F_k$ with $k=i,\ldots,j$. Suppose that $v\in F_{\ell}$, with $\ell>j$. Then $v\in F_i\cap F_{\ell}=F_{i,\ell}$. Therefore $v\in F_{i,\ell}\subsetneq F_{i,j}$  and this is a contradiction since $v\in F_{i,j}\setminus F_{i,j+1}$ and $F_{i,j+1}\supseteq F_{i,\ell}$.\\
By (\ref{statement}), it easily follows  that
\[
\mathcal F=\{ F'_{i,j}\}_{F_{i,j}\in \mathcal B},
\]
is a partition of $V(G)$.
\par\noindent
$(2)$. We prove that $G$ is closed with respect to the labelling induced by the ordering defined in the statement. By Proposition \ref{proposition:newdef} it is sufficient to prove that for every $v\in V(G)$, $N_G^<(v)$, $N_G^>(v)\in \Delta(G)$. Since $v\in V(G)$, then $v\in F'_{i,j}\in \mathcal F$. We claim that $N_G^>(v)\subseteq F_j$, $N_G^<(v)\subseteq F_i$.\\
Let $\{v,w\}\in E(G)$ with $v<w$, we want to prove that $\{v,w\}\subseteq F_j$. Since $v\in F'_{ij}$ by (\ref{statement}) the only cliques containing $v$ are $F_i,\ldots,F_j$.
Therefore, since $\{v,w\}$ is contained in a clique of $G$, then $\{v,w\}\subseteq F_i\cup F_{i+1}\cup\ldots\cup F_j$. By Claim \ref{claim}  $\{v,w\}\subseteq F_i\cup F_j$. Since $v<w$, we have the following cases:
\begin{enumerate}
 \item [(a)] $w\in F_{i,j}'$;
 \item [(b)] $w\in F_{k,\ell}'$, with $k>i$;
 \item [(c)] $w\in F_{k,\ell}'$, with $k=i$ and $\ell>j$.
\end{enumerate}
(a). Obvious. (b) If $w\in F_{k,\ell}'$ with $k>i$, then we have that $w\not\in F_i$, by (\ref{statement}). Hence $w\in F_j$. (c) If $w\in F_{i,\ell}'$ with $\ell>j$, then we have that $w\in F_j$, by (\ref{statement}).

By the same argument we prove that $N_G^<(v)\subset F_i$.
\end{proof}

\begin{cor} \label{Cor1} Let $G$ be a graph.
The following conditions are equivalent:
\begin{enumerate}
 \item[(1)] The graph $G$ is closed on $[n]$.
 \item[(2)] The clique complex $\Delta(G) $ is closed.
 \item[(3)] The binomial edge ideal $J_G$ has a quadratic Gr\"{o}bner basis.
    \end{enumerate}
\end{cor}
\begin{proof} The equivalence follows from Theorems \ref{the:main}, \ref{th:po} and \ref{th:qua}.
\end{proof}
\section*{Acknowledgements} The authors wish to thank Prof.~J\"urgen Herzog for much useful advice and suggestions. We are also grateful for the referee's careful reading and useful comments.

\bibliographystyle{plain}

\end{document}